\documentclass{amsart}      
\usepackage{amsmath,amssymb,amsfonts}
\usepackage{bbm}
\usepackage{graphicx}
\usepackage{mathrsfs}

\newtheorem{theorem}{Theorem}[section]
\newtheorem*{theorem*}{Theorem}
\newtheorem{lemma}[theorem]{Lemma}
\newtheorem*{lemma*}{Lemma}
\newtheorem{corollary}[theorem]{Corollary}
\newtheorem*{corollary*}{Corollary}

\newtheorem*{proposition*}{Proposition}
\theoremstyle{definition} 
\newtheorem{definition}[theorem]{Definition}
\newtheorem*{definition*}{Definition}
\newtheorem*{final_remark}{Final remarks}
\theoremstyle{remark}
\newtheorem{remark}[theorem]{Remark}
\newtheorem*{remark*}{Remark}
\newtheorem{example}[theorem]{Example}
\newtheorem*{example*}{Example}

\newtheorem*{question*}{Question}

\newcommand{\N}{\mathbb N}

\newcommand{\norm}[1]{\left\Vert#1\right\Vert}
\newcommand{\abs}[1]{\left\vert#1\right\vert}

\newcommand{\eps}{\varepsilon}

\newcommand{\ind}{\mathbbm{1}}
\renewcommand{\tilde}{\widetilde}

\newcommand{\dist}{\mathrm{dist}}

\newcommand{\F}{\mathcal{F}}
\newcommand{\G}{{\mathcal G}}

\renewcommand{\epsilon}{\varepsilon}

\renewcommand{\leq}{\leqslant}
\renewcommand{\geq}{\geqslant}

\newcommand{\Ex}[1]{\Upsilon_{#1}}
\newcommand{\mseq}[1]{m_{{#1}}}

\DeclareMathOperator{\Id}{Id}

\begin{document}

\title{Exploding Markov operators}
\thanks{Research is supported from resources for science in years 2013-2018 as research project (NCN grant 2013/08/A/ST1/00275, Poland)}


\author{Bartosz Frej}

\address{Faculty of Pure and Applied Mathematics,\\
Wroc{\l}aw University of Science and Technology,\\ 
Wybrze\.{z}e Wyspia\'{n}skiego 27,\\
50-370 Wroc{\l}aw, Poland \\
							\newline
\includegraphics[width=11pt]{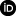}\ %
https://orcid.org/0000-0002-6084-7292
}
\email{Bartosz.Frej@pwr.edu.pl}  

\maketitle

\begin{abstract}
A special class of doubly stochastic (Markov) operators is constructed. These operators come from measure preserving transformations and inherit some of their properties, namely ergodicity and positivity of entropy, yet they may have no pointwise factors.
\end{abstract}

\section{Introduction}
The subject of the current paper lies in the border zone between ergodic theory and operator theory. The main motivation of study was the desire to increase the number of examples of doubly stochastic operators, which escape the scope of classical ergodic theory (because they are not induced by measure preserving maps as their Koopman operators), but they still reveal a nontrivial dynamical behavior. By a \emph{doubly stochastic} or a \emph{Markov operator} we understand an operator $P:L^p(\mu)\to L^p(\nu)$, where $(X,\mu)$ and $(Y,\nu)$ are probability spaces, which fulfills the following conditions:
\begin{itemize}
	\item[(i)] $Pf$ is positive for every positive $f \in L^p(\mu)$,
	\item[(ii)] $P\ind=\ind$ (where $\ind$ is the function constantly equal to 1),
	\item[(iii)] $\int Pf\,d\nu=\int f\,d\mu$ for every $f \in L^p(\mu)$.
\end{itemize} 
For example, the well-studied class of quasi-compact doubly stochastic operators on $L^2$ lies pretty far from the theory of measure preserving maps. But the domain of a quasi-compact operator decomposes into the direct sum of two reducing subspaces, called reversible and almost weakly stable parts, respectively, such that the first one is finite dimensional, while on the other one orbits of functions converge to zero in $L^2$ norm. The restriction of such an operator to the reversible part is Markov isomorphic to a rotation of a compact abelian group (which is finite in this case). The transition probability associated to the operator forces points of the underlying space to ramble periodically through finitely many sets of states (in a fixed order), randomly choosing the succeeding state from a set which is next in the queue. These operators are null, meaning that their sequence entropy is always zero (see \cite{FH} for details). As another example one may think of a convex combination of finitely many measure preserving maps, which leads to studying a rich class of iterated function systems. Unfortunately, such operators are hard to handle by the entropy theory as defined in \cite{DF}---e.g., it is possible that the combination of maps with positive entropy has entropy equal to zero. In the current paper another class of examples which stem from pointwise maps is proposed and some of their properties are investigated.

\section{The definition}
Let $(X,\Sigma,\mu)$ be a standard probability space and let $T:X \to X$ be a measure preserving surjection. Let $(a_n)$ be a strictly decreasing sequence, such that $\sum_{n=1}^\infty a_n =1$. Define a probability distribution $m=\mseq{(a_n)}$ on $\N$  by $m(\{k\}) = a_k$. Now let $Y$ be the Cartesian product $X\times\N$ with the product $\sigma$-algebra and let $\nu=\mu\times m$. For a fixed $k\in\N$ let $\xi_k$ be a partition of $X$ into sets $T^{-k}\{T^{k}x\}$, $x\in X$, and denote by $\xi_k(x)$ an element of the partition $\xi_k$ which contains $x$. 
For every $k$ this partition is measurable (see Appendix 1 in~\cite{FKS} for the precise definition), in other words, the quotient space $X/\xi_k$ is countably separated, so  $X/\xi_k$ is also a standard probability space with measure transported by the map $x\mapsto\xi_k(x)$. Indeed, let $\{B_1,B_2,...\}$ be a separating collection in $X$. If $\xi_k(x)$ and $\xi_k(y)$ are disjoint then $T^{k}x\not=T^{k}y$. Without loss of generality one may assume that $T^{k}x\in B_i$ and $T^{k}y\not\in B_i$ for some $i\in\N$. Then, $\xi_k(z) \subset T^{-k}B_i$ if $T^{k}z\in B_i$ and $\xi_k(z) \cap T^{-k}B_i=\emptyset$ otherwise, so the collection $\{T^{-k}B_i:i\in\N\}$ separates elements of $\xi_k$. Since $T$ is onto, every point in $B_i$ can be represented as $T^{k}z$ for some $z\in X$, so each $T^{-k}B_i=\bigcup_{z\in T^{-k}B_i} \xi_k(z)$ and the collection $\{T^{-k}B_1, T^{-k}B_2,...\}$ separates elements of the partition $\xi_k$. 

Let $\{\mu_C:C\in\xi_k\}$ be the disintegration of $\mu$ over $X/\xi_k$, that is, there is a map $C\mapsto \mu_C$ defined on $X/\xi_k$ with range in the space of all probability measures on $X$, such that each measure $\mu_C$ satisfies $\mu_C(X \setminus C)=0$ and there is a measure $\hat\mu_k$ on $X/\xi_k$ with the property that for any measurable function $f\in L^1(\mu)$,
\[
\int f d\mu = \int_{X/\xi_k} \left(\int_C f(x) d\mu_C(x)\right) d\hat\mu_k(C)
\]
(see \cite{G}). In addition, the map $C\mapsto \mu_C$ is measurable, when the space of probability measures is endowed with the Borel $\sigma$-algebra for the weak${}^*$ topology in the space of probability measures on $X$.
An operator $E_k:L^1(X,\mu)\to L^1(X,\mu)$ given by the formula $E_kf(x) = \int f\circ T \,d\mu_{\xi_k(x)}$ is doubly stochastic. Indeed, it is clear that it is positive and preserves constant functions. Moreover, the function $x\mapsto\mu_{\xi_k(x)}$ is constant on atoms of $\xi_k$, so $E_kf(x)=\int_C f\circ T d\mu_C$ for $x\in C$. Therefore, for every $f\in L^1(X,\mu)$ it holds that
\begin{eqnarray*}
\int E_kf d\mu &=& \int_{X/\xi_k} \left(\int_C E_kf d\mu_C\right) d\hat\mu_k(C) \\
	&=& \int_{X/\xi_k} \left(\int_C f\circ T d\mu_C\right) d\hat\mu_k(C) = \int f\circ Td\mu = \int f d\mu.
\end{eqnarray*}
Define a sequence $(b_k)$ of positive numbers by
\[
b_k=\frac{a_k-a_{k+1}}{a_1}.
\]
As the simplest example one may consider $a_k=b_k=\frac1{2^k}$ or, more generally, $a_k=b_k=(1-a)a^{k-1}$ ($0<a<1$), but other choices are also possible (though if $a_k=b_k$ then one automatically obtains a geometric sequence). Let $\delta_y$ be the Dirac measure concentrated in $y$, that is, $\delta_y(A)=\ind_A(y)$, and let sections of a set $A\subset Y$ and of a function $f$ on $Y$ be denoted by
\newcommand{\sect}[2]{{#1|}_{#2}}
\[
\sect A k =\{x\in X: (x,k)\in A\} \qquad \textrm{and} \qquad \sect f k (x) = f(x,k).
\]

\begin{definition}	\label{def}
An \emph{exploding operator induced by $T$} is a Markov integral operator  on $L^1(Y,\nu)$ defined by
\begin{equation}	\label{explod}
\Ex{T}f (y) = \int f(u) P_T(y,du),
\end{equation}
where the probability kernel $P_T$ is given by:
\begin{eqnarray*}
P_T\big((x,1),A\big) &=& \sum_{k=1}^\infty b_k \mu_{\xi_k(x)}(T^{-1}\sect Ak)\\
P_T\big((x,k),A\big) &=& \delta_{(Tx,k-1)}(A) \qquad \mathrm{for\ } k\geq 2
\end{eqnarray*}
In other words, 
\begin{eqnarray*}
\Ex{T}f(x,1) &=& \sum_{k=1}^\infty b_k \int \sect fk\circ T \, d\mu_{\xi_k(x)} = \sum_{k=1}^\infty b_kE_k\sect fk(x)\\
\Ex{T}f(x,k) &=& f(Tx, k-1) \qquad \mathrm{for\ } k\geq 2
\end{eqnarray*}
\end{definition}

One may visualize the action of $\Ex{T}$ via the transition probablilty $P_T$ in the following way. Each point of $Y$ is a pair consisting of some $x\in X$ and a positive integer $k$. The integer coordinate represents the indication of a clock, which counts down time to an explosion. As long as this indication is greater than 1, the point is mapped according to the action of the pointwise transformation $T$ and the counter goes down by one. When the counter is to be reduced from 1 to 0, the point $x$ explodes and its images are spread over the space (more precisely, over the set of points which would share the common future with $x$, if one considered the evolution by $T$) with counters reset to $k$ with probability~$b_k$.

This class of operators is a generalization of the following example described in \cite{DF}. It has positive entropy, yet it is strictly non-pointwise, meaning that the only pointwise factor of it is the trivial one (see section \ref{pointwise} for definitions).
\begin{example}	\label{example}
Let $(X,\Sigma,\mu)$ consist of the set $X = \{0,1\}^{\N}$ of one-sided 0-1 sequences  with the product $\sigma$-algebra and with the uniform product measure 
$\mu = (\frac 12, \frac 12)^{\N}$. Let $m$ be the
geometric distribution on natural numbers $\N$ given by $m(\{k\}) = 2^{-k}$.

For each finite block $B = (B_1,B_2,\dots,B_k)\in \{0,1\}^k$ let us define the map $\sigma_B:X\to X$ by 
\[
(\sigma_Bx)_n = \left\{\begin{array}{ll}
				 B_n & \mathrm{for}\ \ n\le k\\
				x_{n+1} & \mathrm{for}\ \ n > k
			\end{array}\right..
\]
and define the operator $P$ on $L^1(\mu\times\nu)$  as follows:
\begin{eqnarray*}
Pf(x,k) & = & f(\sigma x, k-1) \ \ \ \mathrm{if} \ k>1,\\
Pf(x,1) & = & \sum_{k=1}^\infty 2^{-k}\sum_{B\in\{0,1\}^k} 2^{-k} f(\sigma_Bx,k).
\end{eqnarray*}
\end{example}

\noindent This example is an instance of our construction if $T$ is the $(\frac12,\frac12)$-Bernoulli shift. We will restore some of its features in the more general case. 

Let us first prove that $\Ex{T}(f)$ defined by \eqref{explod} is Markovian. Clearly, $\Ex{T}f \geq 0$ for $f\geq 0$ and $\Ex{T}$ preserves constants, so the only thing left to show is the fact that $\Ex{T}$ preserves measure. This is checked in the following calculation.
\begin{multline*}
\int_Y \Ex{T}f d\nu = \int_{X\times\{1\}} \Ex{T}f d\nu + \sum_{k=2}^\infty \int_{X\times\{k\}} \Ex{T}f d\nu\\
= a_1 \int_X \sect{(\Ex{T} f)}1 d\mu + \sum_{k=2}^\infty a_k \int_X \sect{(\Ex{T} f)}k d\mu\\
= \sum_{k=1}^\infty a_1 b_k \int_X E_k \sect fk d\mu + \sum_{k=2}^\infty a_k \int_X \sect f{k-1}\circ T \, d\mu\\
= \sum_{k=1}^\infty a_1 b_k \int_X \sect fk d\mu + \sum_{k=1}^\infty a_{k+1} \int_X \sect fk\circ T \, d\mu \\
= \sum_{k=1}^\infty (\underbrace{a_1b_k + a_{k+1}}_{a_k}) \int_X \sect fk d\mu = \int_Y f d\nu
\end{multline*}

\begin{remark}
The above construction is particularly simple for invertible (injective) maps. In this case, $\Ex{T}$ is closely related to the Koopman operator of $T$, precisely, $\Ex{T}f(x,1) = \sum_{k=1}^\infty b_k f(Tx,k)$ or, in other words, $P_T$ sends the point $(x,1)$ to $(Tx,k)$ with probability $b_k$. 
\end{remark}


\section{Pointwise factors}	\label{pointwise}

We begin with several definitions which can be found in \cite{EFHN}.
\begin{definition}
A \emph{unital sublattice} of $L^1(X,\mu)$ is a closed linear subspace of $L^1(X,\mu)$ which contains the constant~$\ind$ and together with each $f$ it contains its conjugate $\bar f$ and its absolute value $|f|$.
\end{definition}
\begin{definition}[\cite{EFHN}, Definition 12.9]	
Let $(X_1,\mu_1)$, $(X_2,\mu_2)$ be a probability space and let $U:L^1(\mu_2)\to L^1(\mu_1)$ be a Markov operator.
\begin{enumerate}
	\item $U$ is a \emph{Markov embedding} if it is a lattice homomorphism (i.e., $|Uf|=U|f|$ for every $f\in L^1(\mu_2)$) or, equvalently, there is a Markov operator $S$ such that $SU$ is an identity.
	\item $U$ is a \emph{Markov isomorphism} if it is a surjective Markov embedding.
\end{enumerate}
\end{definition}
\begin{definition}[\cite{EFHN}, Definition 13.26]	\label{factor}
A \emph{factor} of a doubly stochastic operator $P$ on $L^1(X,\mu)$ is a unital sublattice of $L^1(X,\mu)$, which is invariant under the action of $P$.
\end{definition}
This definition identifies factors as certain subspaces of the domain. It can be shown these subspaces have form $L^1(X,\Sigma_F,\mu)$, where $\Sigma_F$ is a sub-$\sigma$-algebra of the $\sigma$-algebra of all measurable subsets of $X$, which is invariant in this sense that $P\ind_A$ is $\Sigma_F$-measurable for each $A\in\Sigma_F$. Moreover, the representation is unique if one assumes that $\Sigma_F$ is complete with respect to the measure~$\mu$.
Motivated by the theory of classical dynamical systems one may say that if $P_1$ is a doubly stochastic operator on $L^1(\mu_1)$ and $P_2$ is a doubly stochastic operator on $L^1(\mu_2)$ then $P_2$ is a factor of $P_1$ if there is a Markov embedding $U:L^1(\mu_2)\to L^1(\mu_1)$ such that $UP_2 = P_1U$. This definition is consistent with Definition \ref{factor}---the appropriate sublattice is obtained as the image $U(L^1(\mu_2))$.

In case of standard Borel spaces an operator $P_2$ is a factor of $P_1$ if and only if there is a measure-preserving surjection $\pi:X_1\to X_2$ satisfying, for every $f\in L^1(\mu_2)$, the condition $(P_2f)\circ\pi = P_1(f\circ\pi)$. Indeed, a Markov embedding sends characteristic functions of sets to characteristic functions of other sets, hence it defines a homomorphism between measure algebras. In case of standard Borel spaces such homomorphism $\Pi$ is always induced by a pointwise measure preserving map $\pi:X_1\to X_2$ by the formula $\Pi=\pi^{-1}$, so the general definition boils down to the pointwise one.
 Furthermore, if a measure preserving map $T_2$ is a factor of a map $T_1$ (in a classical sense) then the Koopman operator of~$T_2$ is a factor of the Koopman operator of~$T_1$.

A factor of an operator~$P$ is \emph{pointwise} if it is a Koopman operator of a measure preserving map.
This may be understood in two equivalent ways: either there is a sublattice of the form $L^1(X,\Sigma_F,\mu)$, where $\Sigma_F$ is a complete $\sigma$-algebra, which satisfies
\[
A\in\Sigma_F \Rightarrow \exists B\in\Sigma_F\ P\ind_A=\ind_B,
\]
or there is a dynamical system $(Z,\lambda,S)$, where $S:Z\to Z$ is a measure preserving map, and a Markov embedding $U:L^1(\lambda)\to L^1(
\mu)$ with $PUf=U(f\circ S)$ for every $f\in L^1(\lambda)$.

A factor is \emph{trivial} if the measure is concentrated in one point. Obviously, such factors are pointwise and the factor of $P$ is trivial if and only if it can be represented as a subspace $L^1(X,\Sigma',\mu)$, where $\Sigma'$ consists solely of sets of measure zero or one.

Given a map $T$ we define the following relation on $X$:
\begin{equation*}	
x\sim x' \ \Leftrightarrow\ \exists n\in \N\ \ T^n(x)=T^n(x')
\end{equation*}
Equivalently, one can replace the condition above by  $\xi_n(x)=\xi_n(x')$ for some~$n$ or by
\begin{equation}	\label{rel}
\exists i,j\in\N\ \exists z\in X\ \ x\in \xi_i(z) \textrm{\ \ and\ \ } x'\in \xi_j(z).
\end{equation}
Cleary, this is an equivalence relation so it decomposes the space $X$ into disjoint equivalence classes. The equivalence class of $x\in X$ will be denoted by $[x]$ and the corresponding quotient space by $\widetilde{X}$. All such classes are both measurable in $X$ and $\xi_k$-measurable for each $k$.

Consider the space $(\widetilde{X},\widetilde{\Sigma},\widetilde{\mu})$, where $\widetilde{\Sigma}$ and $\widetilde{\mu}$ are the $\sigma$-algebra and the measure transported from $(X,\mu)$ by the canonical map $x\mapsto[x]$.
There is a natural pointwise action $\tilde T$ on $(\widetilde{X},\widetilde{\Sigma},\widetilde{\mu})$ given by $\tilde T[x] = [Tx]$. It is obvious that this is a factor of $T$, but by a straightforward calculation one verifies that it also is a factor of $\Ex{T}$.

\begin{lemma}
Let $L^1(Y,\Sigma_F,\mu)$ be a pointwise factor of $\Ex{T}$. Then for every $A\in\Sigma_F$ the function $\sect{\Ex{T}\ind_A}1$ is a characteristic function of a set being a union of some equivalence classes of $\sim$.
\end{lemma}
\begin{proof}
If $A\in\Sigma_F$ then
\[
\sect{\Ex{T}\ind_A}1 = \sum_{k=1}^\infty b_k \mu_{\xi_k(x)}(T^{-1}\sect{A}k)=\sect{\ind_B}1
\]
for some set $B\in \Sigma_F$. Hence given $x$ one has $\mu_{\xi_k(x)}(\sect{T^{-1}A}k)=0$ for all $k$ or $\mu_{\xi_k(x)}(\sect{T^{-1}A}k)=1$ for all $k$. The value of $\mu_{\xi_k(x)}(A)$ is constant on each element of $\xi_k$. But if $x\sim x'$ then $x$ and $x'$ belong to the same atom of $\xi_i$ for some $i$, so for every $k$ the function $x\mapsto\mu_{\xi_k(x)}(A)$ is also constant on each equivalence class of $\sim$. Thus, $\sect B1$ is a union of equivalence classes on which $\mu_{\xi_k(x)}(A)=1$.\qed
\end{proof}

\begin{theorem}
The exploding operator $\Ex{T}$ has no non-trivial pointwise factor if and only if the $\sigma$-algebra $\widetilde{\Sigma}$  consists solely of sets of measure zero or one.
\end{theorem}
\begin{proof}
If $\widetilde{\Sigma}$ contains a set with measure in $(0,1)$ then $(\widetilde{X},\widetilde{\Sigma},\widetilde{\mu},\tilde T)$ is a non-trivial pointwise factor. If not and if $L^1(Y,\Sigma_F,\mu)$ is a pointwise factor of $\Ex{T}$ then by the above lemma for every $A\in \Sigma_F$ it holds that $\sect{\Ex{T}\ind_{A}}1$ is either equal to zero $\mu$-a.e. or it is equal to one $\mu$-a.e. 
\begin{eqnarray*}
\mu(\sect A1)& = &\nu(\sect A1 \times\N) = \int\ind_{\sect A1 \times\N}d\nu = \int \Ex{T}\ind_{\sect A1 \times\N}d\nu \\
& = & \begin{cases} 
\sum_{k=2}^\infty a_k\mu(T^{-1}(\sect A1)) & \textrm{if }\sect{\Ex{T}\ind_{A}}1=0\ \textrm{$\mu$-a.e.}\\
a_1+\sum_{k=2}^\infty a_k\mu(T^{-1}(\sect A1)) & \textrm{if }\sect{\Ex{T}\ind_{A}}1=1\ \textrm{$\mu$-a.e.}
\end{cases},
\end{eqnarray*}
which is not equal to $\mu(\sect A1)$, unless $\mu(\sect A1)=0$ or $\mu(\sect A1)=1$. In the first case,
\begin{eqnarray*}
\sum_{k=1}^\infty a_k\mu(\sect Ak) & = & \nu(A) = \int \Ex{T}\ind_A d\nu = \int_{X\times\{2,3,...\}}\Ex{T}\ind_A d\nu\\
&=& \sum_{k=2}^\infty a_k\mu(T^{-1}(\sect A{k-1})) = \sum_{k=1}^\infty a_{k+1}\mu(\sect A{k}).
\end{eqnarray*}
Thus, $\sum_{k=1}^\infty (a_k-a_{k+1})\mu(\sect Ak)=0$, so the fact that the sequence $(a_k)$ is strictly decreasing implies that $\mu(\sect Ak)=0$ for all $k$. Performing similar calculations for the case $\mu(\sect A1)=1$, we obtain
\[
\sum_{k=1}^\infty b_k\mu(\sect A{k}) = \sum_{k=1}^\infty \frac{a_k-a_{k+1}}{a_1}\mu(\sect A{k})=1,
\]
which happens only if $\mu(\sect A{k})=1$ for all $k$. Hence, $\nu(A)=0$ or $\nu(A)=1$ for all $A\in\Sigma_F$. \qed
\end{proof}

\begin{corollary}
If $T$ is a one-sided (noninvertible) Bernoulli shift then $\Ex{T}$ is strictly non-pointwise.
\end{corollary}
\begin{proof}
The equivalence relation identifies all points, which are sequences having the same tail. It is not hard to show that any set which belongs to the $\sigma$-algebra defined by such relation is a member of the tail $\sigma$-algebra, therefore it has measure zero or one, by the Kolmogorov's zero-one law.\qed
\end{proof}


\section{Ergodicity}

A doubly stochastic operator is \emph{ergodic} if constant functions are the only functions invariant under the action of the operator.

Let us start with the following operator-theoretic restatement of a classical equivalent definition of ergodicity. Though it is probably well known to the experts, we include the proof.

\begin{theorem}
A doubly stochastic operator $P$ is ergodic if and only if for any non-negative function $f$, which is not equal almost everywhere to zero, the sum $\sum_{k=1}^\infty P^kf$ is positive almost everywhere.
\end{theorem}
\begin{proof}
Assume first that $P$ is ergodic. Let $f$ be a nonnegative function. By the Chacon-Ornstein theorem, the averages $\frac1n \sum_{k=1}^n P^k f$ converge almost everywhere to an invariant function. If $\sum_{k=1}^\infty P^kf=0$ on a set of positive measure then the limit function is zero on this set, hence by ergodicity it must be equal to zero almost everywhere. Since $f$ and the limit function have the same integral, also $f$ must be zero almost everywhere.

Conversely, assume that $f$ is a nonconstant invariant function for $P$. Then $f< \int fd\mu$ on a set of positive measure. For any functions $g$ and $h$ denote by $g\vee h$ the function being the pointwise maximum of $g$ and $h$. Since 
\[
P(f\vee \int f d\mu) \geq Pf\vee \int f d\mu = f\vee \int f d\mu
\]
and both functions have the same integrals, one has $P(f\vee \int f d\mu) = f\vee \int f d\mu$. The function $g=(f\vee \int f d\mu) - \int fd\mu$ is a positive invariant function, which is zero on a set of positive measure. But then $\sum_{k=1}^\infty P^kg$ is zero on the same set, which contradicts our assumption.\qed
\end{proof}

\begin{theorem}
The operator $\Ex{T}$ is ergodic if and only if the map $T$ is ergodic.
\end{theorem}
\begin{proof}
Assume that $T$ is not ergodic, so there is $A\subset X$ such that $0 < \mu(A) <1$ and $T^{-1}A=A$. Clearly, for $k>1$:
\[
\Ex{T}\ind_{A\times\N}(x,k) = \ind_{A\times\N}(Tx,k-1) = \ind_A(Tx) = \ind_{A\times\N}(x,k)
\]
Furthermore, if $x\in A$ then $\xi_k(x)\subset A$, so
\[
\Ex{T}\ind_{A\times\N}(x,1) = \sum_{k=1}^\infty b_k \mu_{\xi_k(x)}(T^{-1}A) = \sum_{k=1}^\infty b_k \mu_{\xi_k(x)}(A) = \sum_{k=1}^\infty b_k = 1.
\]
Therefore, $0\leq \Ex{T} \ind_{A\times\N} \leq \ind_{A\times\N}$. Both these functions have equal integrals, so they must be equal and $\ind_{A\times\N}$ is a non-constant function, which is invariant under the action of $\Ex{T}$. 

For the converse statement, let $T$ be ergodic and let $f$ be a non-negative function, which is not equal almost everywhere to zero. 
If $\sect f1$ is strictly positive on a set $A\subset X$ with $\mu(A)>0$ then, since $\sect f1\circ T$ is constant on $\xi_1(x)$, it holds that 
\[
\Ex{T}f(x,1) \geq b_1 \int \sect f1\circ T \,d\mu_{\xi_1(x)} = b_1 f(Tx,1) >0
\]
for $x\in T^{-1}A$. Inductively, $\Ex{T}^k f(x,1) \geq b_1\Ex{T}^{k-1} f (Tx,1) >0$ for $x\in T^{-k}A$. Ergodicity of $T$ implies that $\bigcup_{n=1}^\infty T^{-n} A$ has measure equal to one, so the sum $\sum_{n=1}^\infty \Ex{T}^nf(x,1)$ is positive $\mu$-almost everywhere. But then it also holds that 
\[
\sum_{n=1}^\infty\Ex{T}^nf(x,k)\geq\sum_{n=k-1}^\infty\Ex{T}^{n-k+1}f(T^{k-1}x,1)>0
\]
for almost every $x$ and every $k>1$.

If $f$ equals zero almost everywhere on $X\times\{1\}$ then $f(x,k)>0$ on a set $A\subset X$ of positive measure $\mu$ for some $k>1$. By definition, 
\[
\Ex{T}f(x,1) \geq b_k \int \sect fk \circ T\, d\mu_{\xi_k(x)}
\]
and the right hand side is positive on a set of positive measure $\mu$. Indeed, 
\[
0<\int \sect fk\,d\mu = \int \sect fk\circ T\, d\mu = \int_{X/_{\xi_k}} \int_C \sect fk\circ T\, d\mu_C d\hat\mu_k(C),
\]
so $\int_C \sect fk\circ T\, d\mu_C$ is positive on a set of positive measure $\hat\mu_k$. Hence, $\Ex{T}f(x,1)$ is positive on a set of positive measure $\mu$ and the hypothesis follows as before.\qed
\end{proof}


\section{Entropy}

For the sequence $(b_n)$ let us denote its $i$th partial sum by $S(i)=\sum_{k=1}^i b_k$ and its $i$th tail by $R(i)=\sum_{k=i+1}^\infty b_k$. We will prove that if $\sum_i R(i)$ converges then the entropy of $\Ex{T}$ is bounded from below by the entropy of $T$. This assumption is satisfied for example by geometric sequences (but not only for them). Let~$\norm{\cdot}_\infty$ denote the norm in $L^\infty(Y,\nu)$.
\begin{lemma}
If $R(i)$ is a summable sequence, then for every measurable set $A\subset X$ the sequence 
\[
\sup_{n\in\N} \norm{\left(\Ex{T}\right)^n\ind_{T^{-i}A\times\N} - \ind_{T^{-(i+n)}A\times\N}}_\infty
\]
converges to $0$, when $i$ goes to infinity.
\end{lemma}
\begin{proof}
Let $i$ be a positive integer. For $k\leq i$ the set $T^{-(i+1)1}A \cap \xi_k(x)$ is nonempty if and only if $T^{k}x\in T^{-(i-k+1)}A$, i.e., $x\in T^{-(i+1)}A$. 
In this case, if $y\in \xi_k(x)$ then $T^k y = T^kx\in T^{-(i-k+1)}A$, so $y\in T^{-(i+1)}A$. 
Therefore, $\xi_k(x) \subset T^{-(i+1)}A$, implying
\[
\mu_{\xi_k(x)}\left(T^{-(i+1)}A\right) = \ind_{T^{-(i+1)}A}(x) \qquad \textrm{for }k\leq i.
\]
Consequently,
\begin{eqnarray*}
S(i)\cdot\ind_{T^{-i-1}A}(x) &\leq& \Ex{T}\ind_{T^{-i}A\times\N} (x,1) \\
&=& \sum_{k=1}^\infty b_k \mu_{\xi_k(x)}(T^{-i-1}A) \leq S(i)\cdot\ind_{T^{-i-1}A}(x) + R(i),
\end{eqnarray*}
hence
\[
\abs{ \Ex{T}\ind_{T^{-i}A\times\N} (x,1) - \ind_{T^{-(i+1)}A\times\N} (x,1)} \leq R(i).
\]
Since $\Ex{T}\ind_{T^{-i}A\times\N}(x,k) = 
\ind_{T^{-(i+1)}A}(x)$ for $k\geq 2$, using the above inequality one gets
\[
\norm{\Ex{T}\ind_{T^{-i}A\times\N} - \ind_{T^{-(i+1)}A\times\N}}_\infty \leq R(i)
\]
for all $i$.

Assume inductively that 
\[
\norm{\Ex{T}^n\ind_{T^{-i}A\times\N} - \ind_{T^{-(i+n)}A\times\N}}_\infty \leq R(i)+...+R(i+n-1)
\]
for some $n$.  Since $\Ex{T}$ is a $L^\infty$ contraction, it holds that
\begin{multline*}
\norm{\Ex{T}^{n+1}\ind_{T^{-i}A\times\N} - \ind_{T^{-i-n-1}A\times\N}}_\infty \leq\\
	\leq \norm{\Ex{T}^{n+1}\ind_{T^{-i}A\times\N} - \Ex{T}\ind_{T^{-i-n}A\times\N}}_\infty + 
	\norm{\Ex{T}\ind_{T^{-i-n}A\times\N} - \ind_{T^{-i-n-1}A\times\N} }_\infty\\
	\leq R(i)+...+R(i+n-1) + R(i+n).
\end{multline*}
Hence, $\norm{\Ex{T}^n\ind_{T^{-i}A\times\N} - \ind_{T^{-i-n}A\times\N}}_\infty \leq \sum_{k=i}^\infty R(k)$ for every $n$, which ends the proof\qed
\end{proof}

The definition of entropy of a doubly stochastic operator is not widely known, so I will devote next few lines for a short introduction to the subject---a detailed exposition may be found in \cite{D} or \cite{DF} and an alternative approach in~\cite{FF}. Similarly to the classical case of the Kolmogorov-Sinai invariant, the entropy of a doubly stochastic operator on $L^1(Y,\nu)$ is defined in several steps. First, the entropy $H_\nu(\F)$ of a finite collection $\F$ of measurable functions with range contained in $[0,1]$ is defined (such collections replace partitions in operator-theoretic definition). Simultaneously, an operation of joining such collections is introduced, for instance, one can define the join of collections $\F$ and $\G$ as a concatenation of $\F$ and $\G$.  Then, the entropy $h_\nu(P,\F)$ of an operator $P$ with respect to a collection $\F$ is obtained as an upper limit (or a limit, if it exists) $\limsup_{n\to\infty} \frac1n H_\mu(\F^n)$, where $\F^n$ stands for the join of $\F,P\F,...,P^{n-1}\F$ and $P^k\F=\{P^kf:f\in\F\}$. Finally, $h_\nu(\F)$ is the supremum $\sup_\F h_\nu(P,\F)$ over all collections under consideration. It was proved in \cite{DF} that any specification of the joining operation
 and the `static' entropy~$H_\nu(\F)$, which satisfies certain set of axioms, leads to a common value of the final notion $h_\nu(P)$. In addition, the conditional entropy of a collection $\F$ with respect to $\G$ is defined as
	\[
	H_\nu(\F|\G) = H_\nu(\F\vee\G)-H_\nu(\G).
	\]

The explicit formula for~$H_\nu(\F)$ will not be needed in the current paper, but I will recall some of the properties of operator entropy, which I use in the forthcomming argument. Both the Kolmogorov-Sinai entropy and the operator entropy will be denoted by the same symbols $H_\nu$ and $h_\nu$. Moreover, the same symbol $\vee$ will be used for both the joining of partitions and the joining of collections of functions---in either case the meaning will be clear from the context. Below there is a list of some of the properties of entropy which can be found in \cite{DF}.
\begin{enumerate}
\renewcommand{\theenumi}{\roman{enumi}}
	\item Let $\xi$ and $\xi'$ be partitions of $Y$ and let $\ind_\xi=\{\ind_A:A\in\xi\}$, $\ind_{\xi'}=\{\ind_A:A\in\xi'\}$. Then $H_\nu(\xi)=H_\nu(\ind_\xi)$ and $H_\nu(\xi\vee\xi')=H_\nu(\ind_\xi \vee \ind_{\xi'})$.
	\item For any finite collections $\F$ and $\G$ it holds that $H_\nu(\F\vee\G) \leq H_\nu(\F)+H_\nu(\G)$. 
	\item For any finite collections $\F_1,...,\F_n$ and $\G_1,...,\G_n$ it holds that
	\[
	H_\nu\left(\bigvee_{i=1}^n \F_i|\bigvee_{i=1}^n \G_i\right) \leq \sum_{i=1}^n H_\nu(\F_i|\G_i)
	\]
	\item Entropy $H_\nu(\F)$ is continuous with respect to $\F$ in the following sense. For two collections $\F=\{f_1,...,f_r\}$ and 
$\G=\{g_1,...,g_r\}$ one defines their $L^1$-distance $\dist(\mathcal F, \mathcal G) $ by a formula
\[
\dist(\mathcal F, \mathcal G)= 
\min_\pi\left\{\max_{1\leq i\leq r}\int |f_i-g_{\pi(i)}|\ d\mu\right\},
\]
where the mi\-ni\-mum ranges over all per\-mu\-ta\-tions $\pi$ 
of a set $\{1,2,\dots r\}$. For every $\eps>0$ there is $\delta>0$ such that if $\F$ and $\G$
have cardinalities at most $r$ and $\dist(\mathcal F, \mathcal G)<\delta$ 
then $|H_\nu(\mathcal F|\mathcal G)|< \eps$.
\end{enumerate}

\begin{theorem}	\label{measure_entropy}
If $R(i)=\sum_{k=i+1}^\infty b_k$ is a summable sequence, then
\[
h_\mu(T) \leq h_\nu(\Ex{T})
\]
\end{theorem}
\begin{proof}
Fix a partition $\xi$ of $X$ and $\epsilon>0$. Denote by $\Id:X\to X$ the identity map on $X$. Given $i\in\N$ one calculates: 
\begin{eqnarray*}
	H_\mu\left(\bigvee_{n=0}^N T^{-i-n}\xi\right) &=& H_\mu\left(\bigvee_{n=0}^N T^n\ind_{T^{-i}\xi}\right) \\
	 & = & H_\nu\left(\bigvee_{n=0}^N (T\times \Id)^n \ind_{T^{-i}\xi\times \N}\right) \\
		&\leq& H_\nu\left(\bigvee_{n=0}^N \Ex{T}^n\ind_{T^{-i}\xi\times \N}\right) +\\
		&& + H_\nu\left(\bigvee_{n=0}^N (T\times \Id)^n \ind_{T^{-i}\xi\times \N} \Big| \bigvee_{n=0}^N \Ex{T}^n\ind_{T^{-i}\xi\times \N}\right)\\
		&\leq& H_\nu\left(\bigvee_{n=0}^N \Ex{T}^n\ind_{T^{-i}\xi\times \N}\right) +\\
		&& + \sum_{n=0}^N H_\nu\left((T\times \Id)^n \ind_{T^{-i}\xi\times \N} \Big| \Ex{T}^n\ind_{T^{-i}\xi\times \N}\right).
\end{eqnarray*}
By the preceding lemma and the continuity of entropy with respect to a collection of functions, the expression under the sum is smaller than $\epsilon$ if only $i$ is big enough. Moreover, for every $i$ one has $h_\mu(T,\xi) = h_\mu(T,T^{-i}\xi)$. Therefore,
\[
h_\mu(T,\xi) \leq h_\nu(\Ex{T},\ind_{T^{-i}\xi\times \N}) + \epsilon \leq h_\nu(\Ex{T}) + \epsilon
\]
and the inequality follows by taking supremum over $\xi$ and infimum over $\epsilon$.\qed
\end{proof}

\begin{final_remark}${}$
\begin{enumerate}
	\item It seems very unlikely that the entropy of $\Ex{T}$ could ever be strictly higher than the entropy of $T$, but, surprisingly, I was not able to prove the equality. However, my conjecture is that the equality holds at least if $T$ is a Bernoulli shift.
	\item Let $X$ be a compact or, more generally, Polish space. An operator $P:C(X)\to C(X)$ is Markov operator if it is positive and preserves constants (in non-compact case $C(X)$ is understood as the space of bounded continuous functions). For a continuous map $T$ our definition~\ref{def} gives a Markov operator on $C(Y)$ if for every continuous $g\in C(X)$ the map $x\mapsto \int g d\mu_{\xi_k(x)}$ is everywhere defined and continuous, i.e., if $x\mapsto \mu_{\xi_k(x)}$ is continuous in the weak${}^*$ topology. It seems reasonable to ask how restrictive are these demands. In \cite{T} one finds an interesting non-classical approach to the idea of disintegration, which yields the same result as the usual disintegration, if well-defined. In particular, theorem 7.1 there gives (together with preceding definitions and construction) a set of assumptions guaranteeing that our definition of $\Ex{T}$ is possible in topological setup. It states that if $X$ and $Z$ are both locally compact and $\sigma$-compact Hausdorff spaces with Radon measures $\mu$ and $\nu$, respectively, $t:X\to Z$ is a continuous map and $Z$ is the support of $\nu$, then the disintegration~$\mu_z$ of $\mu$ along $t$ is defined for all $z\in Z$ and the map $z\to \mu_{z}$ is continuous. In our case, for a given $k$ we consider a partition of $X$ into closed sets $C_{x,k}=T^{-k}T^{k}(x)$ and the role of $Z$ is played by the quotient space $X/_{\xi_k}$. By the definition of identification topology in $X/_{\xi_k}$, this space is a $T_1$-space and a canonical surjection $x\to C_{x,k}$ is continuous (e.g., see \cite{Du}). If this map was open, then compactness of $X$ would imply that $X/_{\xi_k}$ is a Hausdorff space and compactness of  $X/_{\xi_k}$ would follow easily. It is indeed open if $X$ is a subshift---it is easy to see that the image of a cylinder under the identification map is a set in $X/_{\xi_k}$, such that the union of its elements (treated as subsets of $X$) is also a cylinder. Any measure on $X$ with full support transports to a measure with full support on $X/_{\xi_k}$. So the definition~\ref{def} makes sense in topological setup at least in the class of all subshifts having invariant measure with full support. However, to study entropy of this operator one either needs to extend entropy theory introduced in \cite{DF} beyond compact spaces or to define the operator on some compactification of $Y$.
\end{enumerate}
\end{final_remark}


\end{document}